\documentclass[11pt]{amsart}

\usepackage[utf8]{inputenc}
\usepackage[backend=biber,style=alphabetic,maxbibnames=50,maxalphanames=4,
  minalphanames=4]{biblatex}
\usepackage{amssymb,amsfonts}
\usepackage{amsmath,amscd}
\usepackage[all,arc]{xy}
\usepackage{enumerate}
\usepackage{mathrsfs}
\usepackage[pdftex]{graphicx}
\usepackage[T1]{fontenc}
\usepackage[usenames,dvipsnames]{color}
\usepackage[bookmarks=false]{hyperref}
\usepackage{dsfont}
\usepackage{tikz-cd}
\usetikzlibrary{automata,positioning}
\usepackage{fullpage}
\usepackage{wasysym}

\theoremstyle{plain}
\newtheorem{thm}{Theorem}[section]
\newtheorem{cor}[thm]{Corollary}
\newtheorem{prop}[thm]{Proposition}
\newtheorem{lem}[thm]{Lemma}

\theoremstyle{definition}
\newtheorem{defn}[thm]{Definition}

\newtheorem{aDD^+m}[thm]{ADD^+endum}

\theoremstyle{remark}
\newtheorem{rmk}[thm]{Remark}

\title{A thermodynamic path metric for complex H\'enon maps}

\author{Fabrizio Bianchi}
\address{Dipartimento di Matematica, Università di Pisa, Largo Bruno Pontecorvo 5, 56127 Pisa, Italy}
 \email{fabrizio.bianchi$@$unipi.it}
\author{Yan Mary He}
\address{Department of Mathematics\\
	University of Oklahoma\\
	Norman, OK 73019, USA}
\email{he$@$ou.edu}
\date{\today}

\begin{document}

\begin{abstract}
We construct a Hermitian covariance form on hyperbolic components
in parameter spaces
of complex H\'enon maps, associated to the full complex unstable derivative cocycle.
The form measures infinitesimal variations in the marked complex unstable multiplier spectrum. 
Using a recent multiplier rigidity theorem by Cantat--Dujardin, we prove that it induces a distance on every hyperbolic component.

Motivated by Sullivan's dictionary and by the thermodynamic interpretation of the Weil--Petersson metric, our result gives a first higher-dimensional holomorphic-dynamical counterpart of pressure-type metric structures.
On the other hand, the construction differs from the 
one-dimensional theory
in an essential way: it replaces the real geometric potential measuring unstable expansion by the full complex unstable derivative cocycle. 
This also 
suggests a complex derivative cocycle counterpart to pressure-type metric structures in Teichm\"uller theory and Anosov representation theory.
\end{abstract}

\maketitle
\section{Introduction}
The Weil--Petersson metric originates in Teichm\"uller theory and has been a central object in the differential geometry of moduli spaces of Riemann surfaces; see, for instance, the classical works of Ahlfors \cite{Ahlfors61} and Wolpert \cite{Wolpert86}.
Through Sullivan's dictionary \cite{Sullivan83,McMullen95}, holomorphic dynamics, Kleinian groups, and Teichm\"uller theory are connected by a rich network of analogies, with stable components in dynamical parameter spaces playing a role analogous to Teichm\"uller spaces of hyperbolic surfaces. 
Motivated by this analogy and by the thermodynamic interpretation of the Weil--Petersson metric on Teichm\"uller spaces, McMullen \cite{McMullen08} introduced the Weil--Petersson metric for Blaschke products.
The geometry of these metrics for Blaschke products was further studied by Ivrii \cite{Ivrii14}, and more recently by the second author, Lee, and Park \cite{HLP23,HLP25}. For general hyperbolic components in moduli spaces of rational maps, the second author and Nie constructed a Weil--Petersson type Riemannian metric \cite{HeNie23}. This construction has also been extended in our earlier papers \cite{BH24,BH25} to parabolic and Misiurewicz polynomial families. Together, these works are part of a broader program of developing differential-geometric structures on families of holomorphic dynamical systems via thermodynamic formalism.

In this paper, we take a first step toward extending this differential geometric aspect of Sullivan's dictionary to {\it higher-dimensional} holomorphic dynamics, by studying hyperbolic components of complex H\'enon maps.
Beyond holomorphic dynamics, the construction points toward analogous Hermitian covariance forms for complex Anosov representations, in parallel with the pressure metrics of Bridgeman--Canary--Labourie--Sambarino \cite{BCLS15}.
Let $\Omega$ be a hyperbolic component in  
the space $\mathcal H^1_d$
of Hénon maps of algebraic degree $d$, or more generally in
the space $\mathcal H^k_{\mathbf d,\mathbf a}$
of compositions of $k$ Hénon maps
with 
given multidegree and 
multi-Jacobian; see Section 2.1 for the precise definitions. We construct a
\emph{complex multiplier covariance} form
$\langle\cdot,\cdot\rangle_{\rm cmc}$
 on $\Omega$. It is a 
positive semi-definite Hermitian form designed to capture the variation of the complex unstable multiplier data under deformation. 

The construction of the 
complex multiplier covariance form 
differs
in an essential way from the 
related pressure and
Weil--Petersson type
metrics in the one-dimensional theory, which 
are constructed by using a real geometric potential, such as the logarithm of the expansion. For H\'enon maps, however, we replace the real unstable expansion potential by the full complex unstable derivative cocycle, and the standard asymptotic covariance from thermodynamic formalism by an augmented asymptotic covariance. The resulting form is therefore not a pressure form in the usual sense, but rather a Hermitian covariance form on the Liv\v{s}ic class of the complex unstable cocycle.

Thanks to the form 
$\langle\cdot,\cdot\rangle_{\rm cmc}$ and its
positive semi-definiteness,
 we can define
 an associated length
 for piecewise $C^1$-paths
 $\gamma \colon [0,1]\to\Omega$ as
$$
L_{\rm cmc}(\gamma) := \int_0^1
\sqrt{Q(\gamma(t),\dot\gamma(t))}\,dt,
$$
where $Q(\gamma(t),\dot{\gamma}(t)) := \langle \dot{\gamma}(t)^{(1,0)}, \dot{\gamma}(t)^{(1,0)}\rangle_{\rm cmc,\gamma(t)}$ is the real quadratic form associated to the Hermitian form.
This defines a 
path pseudo-distance
$$
d_{\rm cmc}(\lambda_1,\lambda_2):=
\inf_\gamma L_{\rm cmc}(\gamma)
$$
on $\Omega$,
where the infimum is taken over all piecewise $C^1$-paths
$\gamma$ in $\Omega$ joining $\lambda_1$ to $\lambda_2$.
Although the Hermitian form is only positive semi-definite,
and the above infimum
could
a priori vanish also for non-constant paths,
our main theorem shows that the complex unstable multiplier data is rich enough to turn 
$d_{\rm cmc}$
into a non-degenerate global path distance on $\Omega$.

\begin{thm}\label{thm:path-metric}
The path pseudo-distance $d_{\rm cmc}$
is a distance on $\Omega$.
\end{thm}

For hyperbolic families of rational maps,
a multiplier distribution result by Oh--Winter 
\cite{OW17}
allows one to precisely control the variation of the absolute values of the multipliers, and to prove that the corresponding metric is non-degenerate at a given point \cite{HeNie23}. 
In our setting, such a result is not available,
but we can instead
follow the more flexible
approach that we developed
in  \cite{BH24,BH25} to treat
parabolic and Misiurewicz polynomial families, combining it with 
a recent rigidity theorem by Cantat--Dujardin \cite{CD26Henon}. We give an overview of the main ideas in the next section.

\subsection{Strategy of the proof}\label{sec:strategy-pf}
We first briefly describe the construction of the form $\langle\cdot, \cdot\rangle_{\rm cmc}$. 
Fix $\lambda_0\in\Omega$. By structural stability, after passing to a symbolic model $\sigma\colon \Sigma_A\to\Sigma_A$ of the model dynamics of the hyperbolic component, the unstable derivative of $f_\lambda$ defines a complex-valued multiplicative cocycle $A_\lambda\colon \Sigma_A\to\mathbb C^\ast$. After choosing a local branch, write $\zeta_\lambda:=\log A_\lambda$ and
$\phi_\lambda:=\operatorname{Re}\zeta_\lambda=\log |A_\lambda|$.

Denote by $P_{\sigma}$ the topological pressure of $\sigma$ and define $\delta(\lambda)$ to be the unique positive real number satisfying $P_\sigma(-\delta(\lambda)\phi_\lambda) = 0$. Let $\nu_{\lambda_0}$ be the equilibrium state for the potential $-\delta(\lambda_0)\phi_{\lambda_0}$. 
We define the {\it complex multiplier covariance form} at $\lambda_0$
by
$$\langle v,w\rangle_{\mathrm{cmc},\lambda_0}
:= \operatorname{Cov}^{\sharp}_{\nu_{\lambda_0}} \Bigl(
\frac{\partial}{\partial v} \zeta_\lambda \Big|_{\lambda=\lambda_0},
\frac{\partial}{\partial w} \zeta_\lambda \Big|_{\lambda=\lambda_0}
\Bigr), \qquad v,w\in T_{\lambda_0}^{1,0}\Omega.$$
Here $\operatorname{Cov}_{\nu_{\lambda_0}}^{\sharp}$ denotes the augmented asymptotic covariance; see Section \ref{sec:aug-covariance}. 
The form $\langle \cdot, \cdot\rangle_{\rm cmc}$
also 
admits another interpretation as the Hessian form of a real-analytic function. Indeed,
consider the function 
$$G_{\lambda_0}(\lambda) := \operatorname{Cov}^{\sharp}_{\nu_{\lambda_0}} \bigl(\zeta_\lambda-\zeta_{\lambda_0}, \zeta_\lambda-\zeta_{\lambda_0}\bigr).$$
It is not difficult to see that 
$G_{\lambda_0}$
has a local minimum at $\lambda_0$,
which implies that its complex Hessian is well-defined.
One can prove that this complex Hessian recovers the complex multiplier
covariance form.

Geometrically speaking, the complex multiplier covariance form measures infinitesimal variations of the marked complex unstable multiplier spectrum. A 
crucial property is that a
tangent vector has zero norm if and only if the corresponding infinitesimal variation of the complex unstable cocycle is  cohomologous to $0$.
Thanks to this 
characterization, we can prove that
along any path of zero length, the marked unstable multiplier spectrum is constant.
The
recent
multiplier rigidity theorem by
Cantat--Dujardin \cite{CD26Henon} implies that every zero-length path in $\Omega$ is constant.
This plays, in the present 
setting, the role played by the one-dimensional multiplier rigidity
results of Ji--Xie \cite{JiXie23} in our earlier works
\cite{BH24,BH25}. Finally, using the real-analytic dependence of the
form on the tangent bundle, a reduction argument upgrades this
positivity for individual non-constant paths to separation of points
by the induced path pseudo-distance, proving Theorem
\ref{thm:path-metric}.

\subsection{Anosov representations}
We expect this point of view to be useful beyond holomorphic dynamics. Via higher rank Sullivan's dictionary, a natural parallel setting is the theory of Anosov representations, where thermodynamic formalism has  played a central role in the construction of pressure metrics. More specifically,
for Anosov representations into real semisimple Lie groups, Bridgeman--Canary--Labourie--Sambarino \cite{BCLS15} constructed pressure metrics from the thermodynamic formalism applied to real length functions, such as spectral radius or Jordan projection data. 

For holomorphic families of Anosov representations into complex Lie groups, one may instead consider the logarithm of a distinguished complex eigenvalue, or more generally a complex weight cocycle, as a complex-valued H\"older cocycle over the associated flow. The construction in the present paper suggests that one can form a Hermitian covariance form on the Liv\v{s}ic class of this complex cocycle, measuring infinitesimal variations of the marked complex spectral data. This is in the same spirit as the cocycle and cross ratio formalism of Hamenst\"adt \cite{Hamenstadt97,Hamenstadt99} and the pressure metric theory for Anosov representations.

\subsection{Organization of the paper} 
The paper is organized as follows. In Section 2, we recall basic results about H\'enon maps and their moduli spaces. In Section 3, we define the complex multiplier covariance form and characterize its vanishing directions. 
In Section 4,
we prove the analyticity of the form on the tangent bundle
and show
that non-constant piecewise $C^1$-paths have
positive length.
We complete the proof of 
 Theorem \ref{thm:path-metric}
 in Section \ref{s:reduction-proof} 
 by a
 reduction argument, made possible by the analyticity of the form.

\subsection*{Acknowledgments}
This project has received funding from the Programme Investissement d'Avenir (ANR TIGerS, ANR-24-CE40-3604)
and
from the MIUR Excellence Department Project awarded to the Department of Mathematics of the University of Pisa, CUP I57G22000700001.
The first
author is affiliated to the GNSAGA group of INdAM.

\section{Hyperbolic components, structural conjugacies, and unstable cocycles}
\label{s:prelim}

In this section, we recall the basic dynamical properties of H\'enon maps that we will need and fix the notation used
throughout the paper.

\subsection{Parameter spaces of complex H\'enon maps}
By the Friedland--Milnor classification
\cite{FriedlandMilnor89}, every polynomial automorphism of $\mathbb C^2$
of positive entropy
is conjugate to a finite composition of generalized
H\'enon maps.
Following the normal forms
used in \cite{CD26Henon}, fix a multidegree
$\mathbf d=(d_k,\ldots,d_1)\in(\mathbb N_{\ge 2})^k$.
A {\it normalized generalized H\'enon map of multidegree $\mathbf d$} is a
map
$f=h_k\circ\cdots\circ h_1$,
where
$$
h_i(x,y)=(a_i y+p_i(x),x),
\qquad a_i\in\mathbb C^\ast,
$$
and $p_i(x)$ is monic and centered of degree $d_i$. We write
$$
\mathbf a=(a_k,\ldots,a_1)\in(\mathbb C^\ast)^k
$$
for the {\it multi-Jacobian parameter}.

Let $\mathcal H^k_{\mathbf d}$
be the space of normalized H\'enon maps of multidegree $\mathbf d$.
As an algebraic variety,
$\mathcal H^k_{\mathbf d}
        \simeq
        (\mathbb C^\ast)^k
        \times
        \mathbb C^{d_1-1}
        \times\cdots\times
        \mathbb C^{d_k-1}$.
If the multi-Jacobian parameter $\mathbf a$ is fixed, we denote by
$\mathcal H^k_{\mathbf d,\mathbf a} \subset
\mathcal H^k_{\mathbf d}$
the corresponding fixed multi-Jacobian slice. When $k=1$, we write simply
$\mathcal H^1_d$
for the space of normalized H\'enon maps of degree $d$.

\begin{rmk}\label{rmk:finite-conjugacy}
The normal form removes the continuous affine conjugacy freedom. It is not
unique, but the remaining ambiguity is finite.
It
is generated by permutations
of the factors, when compatible with the multidegree, and by a diagonal action
of roots of unity; see \cite[Section 1.2]{CD26Henon}. Thus the {\it moduli space} is a finite quotient of the
normal-form space. Since this residual ambiguity is finite, it creates no
infinitesimal tangent directions. We shall therefore construct the forms on the
smooth normal-form spaces
$\mathcal H^1_d$ or
$\mathcal H^k_{\mathbf d,\mathbf a},
$
and note that the construction descends to the corresponding finite quotient
moduli space.
We observe that, 
when $k=1$, it is not necessary to fix the Jacobian, as no
ambiguity from its decomposition can arise in this case.
\end{rmk}

\subsection{Basic properties of H\'enon maps}
Given a (normalized generalized)
H\'enon map $f$ of positive entropy,
we will denote by
$J_f^\ast$
the closure of the saddle periodic points of
$f$. Equivalently, this is the support of the unique measure of maximal entropy of $f$ \cite{BLSa,BLS93b}.
The measure of maximal entropy is mixing 
\cite{BS3}, see also \cite{BLSa,Dinh05, BD24} 
for more precise results.
It follows immediately that $f|_{J^\ast_f}$ is topologically mixing.

\begin{defn}
A polynomial automorphism $f$ of $\mathbb C^2$ is {\it hyperbolic}
(on $J_f^\ast$)
if $J_f^\ast$ is a uniformly hyperbolic set for $f$, i.e.,
there is a continuous $Df$-invariant splitting
$$
T_{J_f^\ast}\mathbb C^2=E_f^s\oplus E_f^u
$$
and constants $C>0$ and $0<\rho<1$ such that, for all $n\ge 0$,
$$
        \|Df^n|_{E_f^s}\|\le C\rho^n,
        \qquad
        \|Df^{-n}|_{E_f^u}\|\le C\rho^n.
$$
\end{defn}

In this paper, we will be mostly concerned with hyperbolic Hénon maps in the above sense, and  restrict ourselves to the 
dynamics of  $f|_{J^{\ast}_f}$.

\medskip

We observe here that there exists
another natural Julia set
for Hénon maps. Namely, one sets
$
J_f:=J_f^+\cap J_f^-,
$
where
$$
J_f^\pm:=\partial K_f^\pm,
\qquad
K_f^\pm:=\{z\in\mathbb C^2:\{f^{\pm n}(z)\}_{n\ge 0}\text{ is bounded}\},
$$
see \cite{BS91I}.
For maps that are hyperbolic on $J^\ast_f$,
the sets
$J^\ast_f$ and $J_f$ coincide by \cite{Dujardin20}.
It is an open question whether 
$J_f=J_f^\ast$ in general,
see for instance \cite{HenonOpenProblems} for results and references in this direction.

\subsection{Weak stability and hyperbolic components}
We let $\Lambda$ be either  $\mathcal H^1_d$ or $\mathcal H^k_{\mathbf d,\mathbf a}$.
We denote the
corresponding holomorphic family by
$\lambda\mapsto f_\lambda$.
Following Dujardin--Lyubich \cite{DL15}, such a family is called {\it weakly stable} on an open subset of $\Lambda$
if periodic orbits do not bifurcate there, i.e.,
if no
periodic orbit changes type by having an eigenvalue cross the unit circle.
We refer to \cite{DL15} for several equivalence characterizations. 
This
notion is the two-dimensional analogue, in the H\'enon case,
of the $J$-stability for rational maps  on $\mathbb P^1 (\mathbb C)$ \cite{Lyu83typical, MSS83}; see \cite{BBD18}
for the case of higher dimensional endomorphisms
of $\mathbb P^k(\mathbb C)$.
In particular, $\Lambda$ is naturally decomposed into two complementary parts: a 
(weak) stability set, where 
no bifurcations arise, and a bifurcation set.

An important feature of the one-dimensional theory of stability and bifurcation \cite{Lyu83typical,MSS83} is that, as soon as a connected component $\Omega$ of the stability set
contains a hyperbolic rational map, all maps 
in the same component share the same property.
In the case of Hénon maps, this property has been established by
Berger--Dujardin \cite{BergerDujardin2017}.
It thus makes sense to talk about \emph{hyperbolic components}. These are the 
connected components of the 
{\it hyperbolic locus} of $\Lambda$
(the set of maps 
 $f_\lambda$ 
 which are 
 hyperbolic on $J_\lambda^\ast$). Equivalently, it is a connected component of the stability set containing one
hyperbolic parameter
(and hence consisting of hyperbolic parameters).

\medskip

Throughout the paper, we fix 
a hyperbolic component $\Omega\subset\Lambda$, with  $\Lambda=\mathcal H^1_d $ or $\Lambda=\mathcal H^k_{\mathbf d,\mathbf a}$.
Fix a base parameter $\lambda_0\in\Omega$, and write
$f_0:=f_{\lambda_0}$
and
$J_0^\ast:=J_{\lambda_0}^\ast$.
By structural stability of uniformly hyperbolic sets, after possibly shrinking
to a simply connected neighborhood
$U\Subset\Omega$
of $\lambda_0$, there exist conjugacies
$h_\lambda \colon J_0^\ast \to J_\lambda^\ast$ satisfying $h_{\lambda_0}=\operatorname{id}
$ and
$h_\lambda\circ f_0=f_\lambda\circ h_\lambda$.
Moreover, for each $x\in J_0^\ast$, the map
$\lambda\mapsto h_\lambda(x)$
is holomorphic, and the maps $h_\lambda$ are H\"older continuous 
with respect to $x$, uniformly for $\lambda\in U$.

We fix a Markov partition for $f_0|_{J_0^\ast}$. Passing to the associated
symbolic model, we obtain a subshift of finite type
$\sigma \colon \Sigma_A\to\Sigma_A$
and a coding map
$\pi_0 \colon \Sigma_A\to J_0^\ast$
satisfying
$\pi_0\circ\sigma=f_0\circ\pi_0$.
For $\lambda\in U$, set
$\pi_\lambda:=h_\lambda\circ\pi_0.$
Then we have
$\pi_\lambda\circ\sigma=f_\lambda\circ\pi_\lambda$.

\subsection{The unstable bundle and the complex derivative cocycle}
Let $\Omega$ and $U$ be as above.
For each $\lambda\in U$, the hyperbolic splitting gives a fibration
$E_\lambda^u\to J_\lambda^\ast.
$
Pulling back by the coding map gives a complex line bundle
$\pi_\lambda^\ast E_\lambda^u\to \Sigma_A.
$
Since $\Sigma_A$ is totally disconnected, this line bundle is topologically
trivial. 
One reason for working
on the symbolic model is to avoid
potential monodromy
problems that could arise from trying to trivialize the unstable
bundle directly on $J^\ast_\lambda$.

After possibly shrinking $U$, we choose a non-vanishing H\"older section
$e_\lambda(x)\in E_\lambda^u(\pi_\lambda x), x\in\Sigma_A,
$
depending holomorphically on $\lambda$ and H\"older continuously on $x$.

\begin{defn}
The {\it complex unstable derivative cocycle} is the function
$A_\lambda \colon \Sigma_A\to\mathbb C^\ast
$
defined by
$$Df_\lambda(\pi_\lambda x)e_\lambda(x)
=A_\lambda(x)e_\lambda(\sigma x).
$$
\end{defn}

The function $A_\lambda$ is a multiplicative cocycle over $\sigma$. We write
$$
A_\lambda^{(n)}(x) :=
\prod_{j=0}^{n-1} A_\lambda(\sigma^j x).
$$
If $x\in\operatorname{Per}_n(\sigma)$, then
$A_\lambda^{(n)}(x) = \mu_\lambda^u(\pi_\lambda x),
$
where $\mu_\lambda^u(\pi_\lambda x)$ denotes the unstable eigenvalue of
$Df_\lambda^n$ at the saddle periodic point $\pi_\lambda x$.

After possibly 
shrinking $U$, we choose a branch
$\zeta_\lambda:=\log A_\lambda\in C^\alpha(\Sigma_A,\mathbb C)$.
We also set
$\phi_\lambda:=\operatorname{Re}\zeta_\lambda=\log |A_\lambda|$. The following result follows from standard 
graph transform and structural stability arguments.
We give a proof for completeness.

\begin{prop}\label{prop:analytic-cocycle}
After possibly further shrinking $U$, the map $\lambda\mapsto \zeta_\lambda$
is holomorphic from $U$ to $C^\alpha(\Sigma_A,\mathbb C)$. In particular,
$\lambda\mapsto \phi_\lambda
$
is real-analytic from $U$ to $C^\alpha(\Sigma_A,\mathbb R)$.
\end{prop}

\begin{proof}
By structural stability of the hyperbolic set, the conjugacies $ \pi_\lambda=h_\lambda\circ\pi_0\colon \Sigma_A\to J_\lambda^\ast$ depend holomorphically on $\lambda$ pointwise and H\"older continuously on $x$, uniformly for $\lambda$ in a small neighborhood $U$. Moreover, for a holomorphic family of uniformly hyperbolic maps, the stable and unstable directions depend holomorphically on the parameter and H\"older continuously on the base point, with a fixed H\"older exponent. Thus the pulled-back unstable line $ E^u_\lambda(\pi_\lambda x)\subset T_{\pi_\lambda x}\mathbb C^2$ 
depends holomorphically on $\lambda$ and H\"older continuously on $x$. After possibly
shrinking $U$, we choose a local non-vanishing H\"older frame $$ e_\lambda(x)\in E^u_\lambda(\pi_\lambda x), \qquad x\in\Sigma_A, $$ depending holomorphically on $\lambda$ as a $C^\alpha$-section. With respect to this frame, the complex unstable derivative cocycle $A_\lambda$ is defined by $$ Df_\lambda(\pi_\lambda x)e_\lambda(x) = A_\lambda(x)e_\lambda(\sigma x). $$ 
It follows from the above regularity
of $f_\lambda$, $\pi_\lambda$, and $e_\lambda$ 
that $\lambda\mapsto A_\lambda$ is holomorphic from $U$ to $C^\alpha(\Sigma_A,\mathbb C^\ast)$.

It remains to choose the logarithm. Fix a H\"older logarithm
$$
\zeta_{\lambda_0}\in C^\alpha(\Sigma_A,\mathbb C),
        \qquad
e^{\zeta_{\lambda_0}}=A_{\lambda_0}.
$$
Such a logarithm exists for a non-vanishing H\"older function on a subshift of
finite type, after choosing branches on sufficiently small cylinders. Since
$A_\lambda/A_{\lambda_0}\to 1$ in
$C^\alpha(\Sigma_A,\mathbb C)$
as $\lambda\to\lambda_0$, we may shrink $U$ so that
$$
\left\|A_\lambda/A_{\lambda_0}-1\right\|_{C^\alpha}<\frac12
\qquad\text{for all }\lambda\in U.
$$
Then the principal logarithm 
${\rm Log} (\cdot)$
is defined by the convergent power series in the
Banach algebra $C^\alpha(\Sigma_A,\mathbb C)$, and we set
$$
\zeta_\lambda :=
\zeta_{\lambda_0} +
{\rm Log}\left(A_\lambda/A_{\lambda_0}\right).
$$
This gives a holomorphic map
$\lambda\mapsto \zeta_\lambda
\in C^\alpha(\Sigma_A,\mathbb C)$
satisfying $e^{\zeta_\lambda}=A_\lambda$. Therefore
$\phi_\lambda=\operatorname{Re}\zeta_\lambda$
depends real-analytically on $\lambda$ as a
$C^\alpha(\Sigma_A,\mathbb R)$-valued function.
\end{proof}

\begin{rmk}\label{rmk:frame-cobdry}
If we change the H\"older frame by
$$
e_\lambda'(x)=B_\lambda(x)e_\lambda(x),
\qquad B_\lambda(x)\in\mathbb C^\ast,
$$
then the cocycle changes to
$$
A_\lambda'(x) =
A_\lambda(x)\frac{B_\lambda(x)}{B_\lambda(\sigma x)}.
$$
After choosing $b_\lambda=\log B_\lambda$, this gives
$$
\zeta_\lambda'
=
\zeta_\lambda+b_\lambda-b_\lambda\circ\sigma.
$$
Thus changing the frame changes the logarithmic cocycle only by a H\"older
coboundary. In particular, the periodic sums
$S_n\zeta_\lambda(x)=
\zeta_\lambda(x) + \dots + \zeta_\lambda (\sigma^{n-1}(x))$
are independent of the frame and recover
$\log \mu_\lambda^u(\pi_\lambda x)
$
up to an additive constant in $2\pi i\mathbb Z$, whose derivative in parameter
is zero.
\end{rmk}

\subsection{Pressure}
Let $U$ be as above.
For each $\lambda\in U$, define $\delta(\lambda)>0$ by the equation
$P_\sigma(-\delta(\lambda)\phi_\lambda)=0,
$
where $P_\sigma$ denotes topological pressure for the subshift $\sigma$. The next lemma 
in particular shows that $\delta(\lambda)$ is well-defined.

\begin{lem}
For every $\lambda\in U$, there is a unique number
$\delta(\lambda)>0$
satisfying
$P_\sigma(-\delta(\lambda)\phi_\lambda)=0.$
Moreover, the map
$\lambda\mapsto \delta(\lambda)$
is real-analytic on $U$.
\end{lem}

\begin{proof}
The potential $\phi_\lambda$ is cohomologous to the logarithm of the unstable
Jacobian with respect to any smooth Hermitian metric on $E_\lambda^u$.
Uniform hyperbolicity implies that every $\sigma$-invariant probability
measure $\eta$ satisfies
$\int \phi_\lambda\,d\eta>0.
$
It follows that the function
$t\mapsto P_\sigma(-t\phi_\lambda)$
is strictly decreasing. Since
$P_\sigma(0)=h_{\mathrm{top}}(\sigma)>0
$
and
$P_\sigma(-t\phi_\lambda)\to -\infty$ as $t\to+\infty$,
there is a unique zero $\delta(\lambda)>0$.

The map
$(t,\lambda)\mapsto P_\sigma(-t\phi_\lambda)$
is real-analytic,
since the
pressure is real-analytic for real-analytic families
of H\"older potentials over a mixing subshift of finite type 
(see, e.g., \cite{DTUZ26,PU,Parry90})
and $(t,\lambda)\mapsto -t\phi_\lambda$ is real-analytic. 
Moreover,
$$
\partial_t P_\sigma(-t\phi_\lambda)
=
-\int \phi_\lambda\,d\nu_{t,\lambda}<0
$$
at the zero, where $\nu_{t,\lambda}$ is the equilibrium state of
$-t\phi_\lambda$. The real-analytic implicit function theorem gives the
real-analyticity of $\delta(\lambda)$.
\end{proof}

Set
$\delta_0:=\delta(\lambda_0),
\phi_0:=\phi_{\lambda_0}$,
and let $\nu_0$ be the equilibrium state for the potential
$-\delta_0\phi_0$.
Thus we have
$P_\sigma(-\delta_0\phi_0)=0$.

\subsection{Augmented asymptotic covariance and the Livšic criterion}\label{sec:aug-covariance}
Take $g,h\in C^\alpha(\Sigma_A,\mathbb C)$, and write the Birkhoff sum
$S_ng:=\sum_{j=0}^{n-1}g\circ\sigma^j.
$
Define the {\it complex asymptotic covariance} by
$$
\operatorname{Cov}_{\nu_0}(g,h) :=
\lim_{n\to\infty} \frac1n
\int\left(S_ng-n\int g\,d\nu_0\right)
        \overline{
        \left(S_nh-n\int h\,d\nu_0\right)}
        \,d\nu_0
$$
and the
{\it augmented covariance} by
$$
\operatorname{Cov}^{\sharp}_{\nu_0}(g,h) := \left(\int g\,d\nu_0\right)
\overline{\left(\int h\,d\nu_0\right)}+
\operatorname{Cov}_{\nu_0}(g,h).
$$

\begin{lem}\label{lem:Livsic}
For $g\in C^\alpha(\Sigma_A,\mathbb C)$,
$\operatorname{Cov}^{\sharp}_{\nu_0}(g,g)=0$
if and only if there exists a H\"older function
$u\colon \Sigma_A\to\mathbb C$
such that
$g=u-u\circ\sigma$.
In particular, if
$\operatorname{Cov}^{\sharp}_{\nu_0}(g,g)=0$,
then $S_ng(x)=0$
for every $x\in\operatorname{Per}_n(\sigma)$.
\end{lem}

\begin{proof}
We first note that $\operatorname{Cov}^{\sharp}_{\nu_0}(g,g) = \left|\int g\,d\nu_0\right|^2 +
\operatorname{Cov}_{\nu_0}(g,g) = 0$ if and only if both $\int g\,d\nu_0$
and $\operatorname{Cov}_{\nu_0}(g,g)$
vanish.
By the usual variance criterion for mixing subshifts of finite type and H\"older
potentials (see, e.g., \cite{DTUZ26,PU,Parry90}),
$\operatorname{Cov}_{\nu_0}(g,g)=0$
if and only if $g$ is cohomologous to a constant, i.e.,
$g=c+u-u\circ\sigma$ for some H\"older continuous function $u\colon \Sigma_A \to \mathbb C$.
In this case
$\int g\,d\nu_0=c$.
Hence
$\operatorname{Cov}^{\sharp}_{\nu_0}(g,g)=0$
if and only if $g=c+u-u\circ\sigma$ and $c=0$, i.e.,
$g=u-u\circ\sigma$.

Since $g=u-u\circ\sigma$,
the periodic orbit statement follows by summing the coboundary identity along a periodic orbit. This completes the proof.
\end{proof}

\section{The complex multiplier covariance form}

In this section we define the complex multiplier covariance form.
Throughout the section we fix a parameter $\lambda_0\in\Omega$ and a
sufficiently small simply connected neighbourhood $U\Subset\Omega$ of
$\lambda_0$. We use the notation of Section 
\ref{s:prelim}: on the fixed symbolic
model $\sigma:\Sigma_A\to\Sigma_A$ we have a holomorphic family of
logarithmic unstable cocycles
$$
  \zeta_\lambda=\log A_\lambda\in C^\alpha(\Sigma_A,\mathbb C),
  \qquad
  \phi_\lambda=\operatorname{Re}\zeta_\lambda=\log |A_\lambda|.
$$
For simplicity, we set $\delta_0:=\delta(\lambda_0)$, $\phi_0:=\phi_{\lambda_0}$,
and denote by $\nu_0$ the equilibrium state of
$-\delta_0\phi_0$.

\subsection{Definition and well-definedness}

Since $\lambda\mapsto\zeta_\lambda$ is holomorphic as a
$C^\alpha(\Sigma_A,\mathbb C)$-valued map, every vector
$v\in T^{1,0}_{\lambda_0}\Omega$ defines an element
$$
\partial_v\zeta_{\lambda_0}:=
\frac{\partial}{\partial v}\zeta_\lambda \Big|_{\lambda=\lambda_0}
\in C^\alpha(\Sigma_A,\mathbb C).$$
Equivalently, if $\lambda(t)$ is a holomorphic curve with
$\lambda(0)=\lambda_0$ and $\lambda'(0)=v$, then
$$
  \partial_v\zeta_{\lambda_0}
  =
  \left.\frac{d}{dt}\right|_{t=0}\zeta_{\lambda(t)} .
$$
In local holomorphic coordinates $(\lambda_1,\ldots,\lambda_N)$, if
$v=\sum_j v_j\partial/\partial\lambda_j$, then
$$
  \partial_v\zeta_{\lambda_0}
  =
  \sum_{j=1}^N v_j
  \left.\frac{\partial\zeta_\lambda}{\partial\lambda_j}
  \right|_{\lambda=\lambda_0}.
$$

\begin{rmk}\label{r:per-sn}
Observe that if $x\in\operatorname{Per}_n(\sigma)$ and
$p_\lambda:=\pi_\lambda(x)$
is the corresponding holomorphic continuation of the saddle periodic point,
then the Birkhoff sum $S_n(\partial_v\zeta_{\lambda_0})(x)$ gives the infinitesimal variation of the logarithms of unstable multipliers, i.e.,
$$
S_n(\partial_v\zeta_{\lambda_0})(x) = \partial_v\log \mu^u_\lambda(p_\lambda) \big|_{\lambda=\lambda_0}.
$$
The ambiguity in the logarithm is an additive constant in
$2\pi i\mathbb Z$, hence has zero derivative.
\end{rmk}

We now define the {\it complex multiplier covariance form}.

\begin{defn}\label{def:cmc}
For $v,w\in T^{1,0}_{\lambda_0}\Omega$, define
$$
\langle v,w\rangle_{\rm cmc,\lambda_0} := \operatorname{Cov}^{\sharp}_{\nu_0}\bigl(\partial_v\zeta_{\lambda_0},   \partial_w\zeta_{\lambda_0}\bigr). 
$$
\end{defn}

\begin{prop}
The complex multiplier covariance form is independent of the local
branch of $\log A_\lambda$, of the Hölder frame of the pulled-back
unstable bundle, and of the Markov coding used to represent
$f_{\lambda_0}|_{J^\ast_{\lambda_0}}$. Hence it is intrinsically defined on
$T^{1,0}_{\lambda_0}\Omega$.
\end{prop}

\begin{proof}
Changing the branch of $\log A_\lambda$ on the connected
neighbourhood $U$ changes $\zeta_\lambda$ by a constant in
$2\pi i\mathbb Z$, independent of $\lambda$. Therefore
$\partial_v\zeta_{\lambda_0}$ is unchanged.

We next change the Hölder frame. If
$e'_\lambda(x)=B_\lambda(x)e_\lambda(x)$, then, after choosing a local
logarithm $b_\lambda=\log B_\lambda$, the logarithmic cocycle changes
by
$\zeta'_\lambda
= \zeta_\lambda+b_\lambda-b_\lambda\circ\sigma$ by Remark \ref{rmk:frame-cobdry}.
Hence
$$
\partial_v\zeta'_{\lambda_0}
= \partial_v\zeta_{\lambda_0}
  +\partial_v b_{\lambda_0}
  -\partial_v b_{\lambda_0}\circ\sigma .
$$
The added term is a Hölder coboundary. It has zero integral and zero
asymptotic covariance with every Hölder function. Therefore, we have
$$
\operatorname{Cov}^{\sharp}_{\nu_0}
\bigl(\partial_v\zeta_{\lambda_0}',
\partial_w\zeta_{\lambda_0}'\bigr)
=
\operatorname{Cov}^{\sharp}_{\nu_0}
\bigl(\partial_v\zeta_{\lambda_0},
\partial_w\zeta_{\lambda_0}\bigr),$$
which gives the desired invariance.

Finally, two Markov codings admit a common symbolic refinement. Passing
to the refinement pulls back the cocycle, the equilibrium state and the
covariance form, and these quantities are unchanged under such a
finite-to-one symbolic refinement. The periodic sums of
$\zeta_\lambda$ are intrinsic, since they are the logarithms of the
unstable multipliers. Thus the form computed using any Markov coding is
the same.
\end{proof}

\subsection{Positivity and null directions}
\begin{prop}\label{p:pos-null}
The complex multiplier covariance form is Hermitian and positive
semi-definite. Moreover, for $v\in T^{1,0}_{\lambda_0}\Omega$, we have
$$
  \langle v,v\rangle_{{\rm cmc},\lambda_0}=0
\qquad \Leftrightarrow
\qquad 
\partial_v\zeta_{\lambda_0}
  =
  u-u\circ\sigma$$
for some
H\"older continuous
function $u
\colon \Sigma_A\to\mathbb C$.
\end{prop}

\begin{proof}
Hermitian symmetry and positive semi-definiteness follow from the
corresponding properties of the augmented covariance
$\operatorname{Cov}^{\#}_{\nu_0}$.
The nullity statement follows by applying
Lemma~\ref{lem:Livsic} 
to
$g=\partial_v\zeta_{\lambda_0}$. Thus
$\langle v,v\rangle_{{\rm cmc},\lambda_0}=0$ if and only if
$\partial_v\zeta_{\lambda_0}$ is a H\"older coboundary. 
\end{proof}

\begin{cor}\label{c:per}
If $\langle v,v\rangle_{{\rm cmc},\lambda_0}=0$, then
for every $x\in\operatorname{Per}_n(\sigma)$, we have
$$
S_n(\partial_v\zeta_{\lambda_0})(x)=0.
$$
Equivalently, every marked unstable multiplier has zero infinitesimal
variation in the direction $v$.
\end{cor}

\begin{proof}
Summing the
coboundary identity
given by Proposition \ref{p:pos-null}
over a periodic orbit gives
$S_n(\partial_v\zeta_{\lambda_0})(x)=0$. 
By Remark \ref{r:per-sn}, 
this is the infinitesimal vanishing of the corresponding
logarithmic unstable multiplier.
\end{proof}

\subsection{A Hessian interpretation}
Here we give an  interpretation of the complex multiplier covariance
form as the Hessian of a non-negative real-analytic function. This
interpretation will not be used in the proof of Theorem \ref{thm:path-metric},
where the covariance form itself is
enough, but helps to link
our construction to those in \cite{HeNie23,BH24,BH25}, 
where a Hessian definition is the starting point.

We fix $\lambda_0\in\Omega$ and use the notation of the previous
subsections. For $\lambda$ in a sufficiently small simply connected
neighbourhood $U$ of $\lambda_0$, set
$q_\lambda:=\zeta_\lambda-\zeta_{\lambda_0}$ and
define
$$
G_{\lambda_0}(\lambda) :=
\operatorname{Cov}^{\#}_{\nu_0} (q_\lambda,q_\lambda).
$$
Thus $G_{\lambda_0}\colon U\to\mathbb R$ is obtained by measuring the variation of the
complex unstable derivative cocycle from the base parameter
$\lambda_0$, relative to the equilibrium state
$\nu_0$.

\begin{prop}
The function $G_{\lambda_0}$ is independent of the choices of the
local branch of $\log A_\lambda$ and of the Hölder frame of the
pulled-back unstable bundle. It is real-analytic, satisfies
$G_{\lambda_0}\geq 0$, and has a local minimum at $\lambda_0$.
Moreover, for every
$v,w\in T^{1,0}_{\lambda_0}\Omega$, we have
$$
\partial_v\overline{\partial}_w G_{\lambda_0}(\lambda_0)
  =
\langle v,w\rangle_{{\rm cmc},\lambda_0}.
$$
\end{prop}

\begin{proof}
Changing the branch of $\log A_\lambda$ changes both
$\zeta_\lambda$ and $\zeta_{\lambda_0}$ by the same constant in
$2\pi i\mathbb Z$, hence leaves $q_\lambda$ unchanged.
If we replace the Hölder frame by
$e'_\lambda=B_\lambda e_\lambda$, and choose a local logarithm
$b_\lambda=\log B_\lambda$, then
$
\zeta'_\lambda =
\zeta_\lambda+b_\lambda-b_\lambda\circ\sigma 
$ by Remark \ref{rmk:frame-cobdry}.
Therefore, we have
$$
q'_\lambda =
q_\lambda+r_\lambda-r_\lambda\circ\sigma,
\qquad \text{where} \quad
r_\lambda:=b_\lambda-b_{\lambda_0}.
$$
The added term is a Hölder coboundary. It has zero integral and zero
asymptotic covariance with every Hölder function. Hence, we have
$$
\operatorname{Cov}^{\#}_{\nu_0}(q'_\lambda,q'_\lambda)
  =
\operatorname{Cov}^{\#}_{\nu_0}(q_\lambda,q_\lambda).
$$
This proves the independence of the frame.

Since $\lambda\mapsto\zeta_\lambda$ is holomorphic as a
$C^\alpha(\Sigma_A,\mathbb C)$-valued map and
$\operatorname{Cov}^{\#}_{\nu_0}$ is a continuous Hermitian
sesquilinear form on $C^\alpha(\Sigma_A,\mathbb C)$, the function
$G_{\lambda_0}$ is real-analytic. It is non-negative by positivity of
the augmented covariance, and $G_{\lambda_0}(\lambda_0)=0$. Hence
$\lambda_0$ is a local minimum.

It remains to compute the complex Hessian. Since
$q_{\lambda_0}=0$ and $\lambda\mapsto q_\lambda$ is holomorphic,
all terms involving $q_{\lambda_0}$ vanish when differentiating twice
at $\lambda_0$. Thus,
$$
\partial_v\overline{\partial}_w \operatorname{Cov}^{\#}_{\nu_0}(q_\lambda,q_\lambda)
  \big|_{\lambda=\lambda_0}
  =
\operatorname{Cov}^{\#}_{\nu_0}
(\partial_v\zeta_{\lambda_0},  \partial_w\zeta_{\lambda_0}).
$$
By Definition \ref{def:cmc}, 
the right-hand side is precisely
$\langle v,w\rangle_{{\rm cmc},\lambda_0}$.
\end{proof}

\begin{rmk}
The above definition is given
on the fixed symbolic model $\Sigma_A$.
This can be viewed as a symbolic version of
the geometric formulation using the pushed-forward
equilibrium state as in \cite{HeNie23,BH24,BH25}.
Indeed, the structural conjugacy
$h_\lambda\colon J_{\lambda_0}^\ast\to J_\lambda^\ast$
pushes $\nu_0$ forward to an $f_\lambda$-invariant probability measure
$\nu_{\lambda_0,\lambda}:=(h_\lambda)_*\nu_0
$
on $J_\lambda^\ast$.
In symbolic coordinates, this corresponds to
integrating $\zeta_\lambda$ against the fixed measure $\nu_0$ on
$\Sigma_A$.
Thus,
%
defining 
the {\it complex unstable Lyapunov functional based at
$\lambda_0$} as
$$
\operatorname{Ly}^{u,\mathbb C}_{\lambda_0}(\lambda):=\int_{\Sigma_A}\zeta_\lambda\,d\nu_0,
$$ 
we have
\begin{align*}
\operatorname{Cov}^{\sharp}_{\nu_0}(\zeta_\lambda-\zeta_{\lambda_0},\zeta_\lambda-\zeta_{\lambda_0})
& = \left|\int_{\Sigma_A} \zeta_\lambda-\zeta_{\lambda_0}\,d\nu_0\right|^2
+\operatorname{Cov}_{\nu_0}(\zeta_\lambda-\zeta_{\lambda_0},\zeta_\lambda-\zeta_{\lambda_0})\\
& = \left|\operatorname{Ly}^{u,\mathbb C}_{\lambda_0}(\lambda) -
\operatorname{Ly}^{u,\mathbb C}_{\lambda_0}(\lambda_0)\right|^2
+\operatorname{Cov}_{\nu_0}(\zeta_\lambda-\zeta_{\lambda_0},\zeta_\lambda-\zeta_{\lambda_0}).
\end{align*}
\end{rmk}

\section{Analyticity of the complex multiplier covariance form and length vanishing}

In this section we prove the two main
dynamical properties of the complex multiplier covariance form that will be used to obtain the path metric. First, we prove that the associated quadratic form on the real tangent bundle is real-analytic. Then we show that a path whose tangent vectors have zero norm has constant marked unstable multiplier spectrum. The multiplier rigidity theorem of Cantat--Dujardin 
\cite{CD26Henon}
then implies that every non-constant piecewise $C^1$-path has positive length.

\subsection{Analyticity on the tangent bundle}
Let $T_{\mathbb R}\Omega$ denote the real tangent bundle of $\Omega$. If
$\xi\in T_{\mathbb R,\lambda}\Omega$, we write
$\xi^{1,0}\in T^{1,0}_\lambda\Omega$
for its $(1,0)$-part,
 so that in local holomorphic coordinates
\(z_j=x_j+iy_j\), if
\[
\xi=\sum_j a_j\frac{\partial}{\partial x_j}
    +b_j\frac{\partial}{\partial y_j},
\qquad
\text{then}
\qquad
\xi^{1,0}=\sum_j(a_j+ib_j)\frac{\partial}{\partial z_j}.
\]
We define the real quadratic form associated to the
complex multiplier covariance form by
$$
Q(\lambda,\xi) :=
\langle \xi^{1,0},\xi^{1,0}\rangle_{\rm cmc,\lambda}.
$$
Thus $Q$ is a non-negative quadratic form on $T_{\mathbb R}\Omega$.

In local holomorphic coordinates, set
$$
\dot\zeta_{\lambda,\xi} :=
\partial_{\xi^{1,0}}\zeta_\lambda \in C^\alpha(\Sigma_A,\mathbb C).
$$
Since $\lambda\mapsto\zeta_\lambda$ is holomorphic, if
$\gamma(t)$ is a real $C^1$-path with
$\gamma(0)=\lambda$
and $\dot\gamma(0)=\xi$,
then
$$
\left.\frac{d}{dt}\right|_{t=0}\zeta_{\gamma(t)}
=\partial_{\xi^{1,0}}\zeta_\lambda
=\dot\zeta_{\lambda,\xi},   
\qquad
\text{and}
\qquad
Q(\lambda,\xi) =
\operatorname{Cov}^{\sharp}_{\nu_\lambda}\bigl(\dot\zeta_{\lambda,\xi}, \dot\zeta_{\lambda,\xi}\bigr),
$$
where $\nu_\lambda$ is the equilibrium state for
$-\delta(\lambda)\phi_\lambda$.

The next proposition proves the analyticity of the form on the tangent bundle. 
We remark that 
in \cite{BH24,BH25}
the analyticity is proved by a different
method, based on the extension techniques of 
\cite{SU10,UZ04real}.

\begin{prop}\label{prop:analyticity-tangent-bundle}
The map
$Q\colon T_{\mathbb R}\Omega\to \mathbb R_{\ge 0}$ given by
$(\lambda,\xi)\mapsto Q(\lambda,\xi)$
is real-analytic. In particular, its restriction to the unit tangent bundle of
$\Omega$, with respect to any auxiliary real-analytic Riemannian metric, is
real-analytic.
\end{prop}

\begin{proof}
We work in a holomorphic coordinate neighborhood
$U\Subset\Omega$. By Proposition \ref{prop:analytic-cocycle}, the map
$\lambda\mapsto \zeta_\lambda
$ is holomorphic as a map from $U$ to
$C^\alpha(\Sigma_A,\mathbb C)$.
Therefore, the map
$$(\lambda,\xi)\mapsto \dot\zeta_{\lambda,\xi} =
\partial_{\xi^{1,0}}\zeta_\lambda
$$
is real-analytic from $T_{\mathbb R}U$ to $C^\alpha(\Sigma_A,\mathbb C)$. 

It remains to justify the real-analytic dependence of the covariance term.
Set
$F_\lambda:=-\delta(\lambda)\phi_\lambda.
$
Then $F_\lambda$ is a real-analytic family of real-valued H\"older potentials,
and
$P_\sigma(F_\lambda)=0.
$
Let $\mathcal L_\lambda$ be the associated Ruelle operator
$$
\mathcal L_\lambda u(x):=
\sum_{\sigma (y)=x}e^{F_\lambda(y)}u(y).
$$
Since $F_\lambda$ depends real-analytically on $\lambda$, the family
$\mathcal L_\lambda \colon C^\alpha(\Sigma_A)\to C^\alpha(\Sigma_A)$ depends
real-analytically on $\lambda$. The leading eigenvalue is simple and isolated.
Hence the leading eigenprojection, the equilibrium state $\nu_\lambda$, and
the correlation functions depend real-analytically on $\lambda$.

To make the dependence of the covariance explicit, consider the pressure function $$ \mathcal P(s_1,s_2,\lambda,\xi) := P_\sigma\left( F_\lambda +s_1\operatorname{Re}\dot\zeta_{\lambda,\xi} +s_2\operatorname{Im}\dot\zeta_{\lambda,\xi} \right). $$ It is jointly real-analytic for $(s_1,s_2)$ near $(0,0)$ and $(\lambda,\xi)\in T_{\mathbb R}U$. Its first derivatives at $(s_1,s_2)=(0,0)$ give $$ \int \operatorname{Re}\dot\zeta_{\lambda,\xi}\,d\nu_\lambda, \qquad \int \operatorname{Im}\dot\zeta_{\lambda,\xi}\,d\nu_\lambda,$$
and its second derivatives give the real asymptotic covariances of $\operatorname{Re}\dot\zeta_{\lambda,\xi}$ and $\operatorname{Im}\dot\zeta_{\lambda,\xi}$. Since the Hermitian covariance of a complex-valued function is determined by the real covariance matrix of its real and imaginary parts, it follows that the map
$$ (\lambda,\xi)\mapsto \operatorname{Cov}_{\nu_\lambda} (\dot\zeta_{\lambda,\xi},\dot\zeta_{\lambda,\xi}) $$ is real-analytic.
 The augmented term $$ \left|\int \dot\zeta_{\lambda,\xi}\,d\nu_\lambda\right|^2 $$ is real-analytic as well. Therefore, the function $$ (\lambda,\xi)\mapsto \operatorname{Cov}^{\#}_{\nu_\lambda} (\dot\zeta_{\lambda,\xi},\dot\zeta_{\lambda,\xi}) $$ is real-analytic, which gives the real-analyticity of $Q$.
\end{proof}

\subsection{Null directions along a path}
Let $\gamma \colon I\to\Omega$ be a $C^1$-path and fix $s_0\in I$. After
shrinking $I$ if necessary, structural stability gives conjugacies
$h_{s,s_0}\colon J_{\gamma(s_0)}^\ast\to J_{\gamma(s)}^\ast$
satisfying
$h_{s,s_0}\circ f_{\gamma(s_0)}=
f_{\gamma(s)}\circ h_{s,s_0}$.
Using the symbolic model for $f_{\gamma(s_0)}|_{J_{\gamma(s_0)}^\ast}$,
we pull back the logarithmic unstable cocycle of $f_{\gamma(s)}$ to a
H\"older function
$\zeta_{\gamma(s)}^{(s_0)}\in C^\alpha(\Sigma_A,\mathbb C).$
We define 
$$K_\gamma(s,s_0):=
\zeta_{\gamma(s)}^{(s_0)}-
\zeta_{\gamma(s_0)}^{(s_0)}.$$

For $x\in\operatorname{Per}_n(\sigma)$, if $p_s$ denotes the corresponding
continuation of the saddle periodic point, then
$$
S_nK_\gamma(s,s_0)(x) = \log \mu^u_{\gamma(s)}(p_s)
- \log \mu^u_{\gamma(s_0)}(p_{s_0})
\quad \mod 2\pi i \mathbb Z.
$$
Thus $K_\gamma(s,s_0)$ is the change of the actual complex unstable
multiplier spectrum.

\begin{lem}\label{lem:5.2}
Let $\gamma \colon I\to\Omega$ be a $C^1$-path. Suppose that
$Q(\gamma(s),\dot\gamma(s))=0$ for some $s\in I$. Then
$\left.\frac{\partial}{\partial t}\right|_{t=s}
K_\gamma(t,s)$
is a H\"older coboundary over $\sigma$. In particular, for every
$x\in\operatorname{Per}_n(\sigma)$, we have
$$
\left.\frac{\partial}{\partial t}\right|_{t=s}
S_nK_\gamma(t,s)(x)=0.
$$
\end{lem}

\begin{proof}
We work with the symbolic model based at the parameter $\gamma(s)$. By definition, we have
$$ Q(\gamma(s),\dot\gamma(s)) = \operatorname{Cov}^{\#}_{\nu_{\gamma(s)}} \left( \left.\frac{\partial}{\partial t}\right|_{t=s} \zeta^{(s)}_{\gamma(t)}, \left.\frac{\partial}{\partial t}\right|_{t=s} \zeta^{(s)}_{\gamma(t)} \right). $$ Since this quantity is zero, Proposition~\ref{p:pos-null} 
applied at the parameter $\gamma(s)$ to the vector $\dot\gamma(s)^{1,0}$ implies that
$\left.\frac{\partial}{\partial t}\right|_{t=s} \zeta^{(s)}_{\gamma(t)}$
is a Hölder coboundary. Since $K_\gamma(t,s)=\zeta^{(s)}_{\gamma(t)} -\zeta^{(s)}_{\gamma(s)}$,
the same is true for $$ \left.\frac{\partial}{\partial t}\right|_{t=s} K_\gamma(t,s). $$ Summing this coboundary over a periodic orbit gives the assertion; see Corollary \ref{c:per}.
\end{proof}

\begin{cor}\label{cor:multiplier-constant}
Let $\gamma\colon I\to\Omega$ be a $C^1$-path. If
$Q(\gamma(s),\dot\gamma(s))=0$ for every $s\in I$,
then every marked saddle periodic point has constant unstable multiplier along
$\gamma$. That is, for every saddle periodic point $p_{s_0}$ of
$f_{\gamma(s_0)}$, its continuation $p_s$ satisfies
$\mu^u_{\gamma(s)}(p_s)=\mu^u_{\gamma(s_0)}(p_{s_0})$
for every $s\in I$.
\end{cor}

\begin{proof}
Fix $s_0\in I$, and let $p_{s_0}$ be a saddle periodic point of
$f_{\gamma(s_0)}$ of period $n$. Let $p_s$ denote its continuation along $\gamma$. 
It is enough to show that the map $s\mapsto \mu^u_{\gamma(s)}(p_s)$
is constant. Let $s\in I$ be arbitrary.
We apply Lemma \ref{lem:5.2} with base time $s$. Choose the symbolic model based at $f_{\gamma(s)}|_{J^\ast_{\gamma(s)}}$,
and let $x\in \operatorname{Per}_n(\sigma)$ code the marked periodic point
$p_s$. Lemma \ref{lem:5.2} gives
$$
        \frac{\partial}{\partial t}\bigg|_{t=s}
        S_nK_\gamma(t,s)(x)=0 .
$$
On the other hand, by the definition of $K_\gamma(t,s)$, we have
$$
        S_nK_\gamma(t,s)(x)
        =
        \log \mu^u_{\gamma(t)}(p_t)
        -
        \log \mu^u_{\gamma(s)}(p_s)
        \quad \mod 2\pi i\mathbb Z .
$$
After choosing a local branch of the logarithm near $s$, the ambiguity in
$2\pi i\mathbb Z$ is locally constant and hence has zero derivative.
Therefore, we have
$$
        \frac{d}{dt}\bigg|_{t=s}
        \log \mu^u_{\gamma(t)}(p_t)=0 .
$$
Since $s\in I$ was arbitrary, the function
$s \mapsto \log \mu^u_{\gamma(s)}(p_s)$
has zero derivative locally along $I$. Hence it is constant on $I$, and so
$\mu^u_{\gamma(s)}(p_s)=\mu^u_{\gamma(s_0)}(p_{s_0})
$
for every $s\in I$. Since the marked periodic point was arbitrary, every
marked saddle multiplier is constant along $\gamma$.
\end{proof}

\subsection{Positive length of non-constant piecewise $C^1$-paths}
For a piecewise $C^1$-path $\gamma \colon [0,1]\to\Omega$, define its
complex multiplier covariance
length
by
$$
L_{\rm cmc}(\gamma) := \int_0^1
\sqrt{Q(\gamma(t),\dot\gamma(t))}\,dt.
$$

\begin{prop}\label{prop:positive-length}
Every non-constant piecewise $C^1$-path
$\gamma\colon [0,1]\to\Omega$
satisfies
$L_{\rm cmc}(\gamma)>0$.
\end{prop}

\begin{proof}
Suppose, by contradiction, that
$L_{\rm cmc}(\gamma)=0.
$
Since the integrand is continuous and non-negative on each $C^1$-piece, we
have
$Q(\gamma(t),\dot\gamma(t))=0$
for every $t$ on each smooth piece of the path.

By Corollary \ref{cor:multiplier-constant}, the
marked unstable multiplier spectrum is constant along $\gamma$. Hence,
 the image $\gamma([0,1])$ is contained in a single fiber of the marked
unstable multiplier spectrum.
By the
multiplier rigidity theorem of Cantat--Dujardin \cite{CD26Henon}, this fiber is
finite. 
Since $\gamma([0,1])$ is connected
and contained in a finite set, it consists of a single point. Hence $\gamma$
is constant, giving a contradiction.
\end{proof}

\section{A reduction argument and proof of Theorem \ref{thm:path-metric}}\label{s:reduction-proof}

We now pass from positivity of individual non-constant
piecewise $C^1$-paths to separation of points by the induced path pseudo-distance. This step is necessary since, a priori, the infimum of the lengths of paths joining two distinct points could be zero even if no individual non-constant path has zero length. The analyticity of $Q$ rules this out by a reduction argument as in \cite[Section 5.4]{BH24}.

\begin{lem}\label{lem:semi-reduction}
Let $M$ be a real-analytic manifold and let $Q$ be a real-analytic
positive semi-definite quadratic form on $TM$. Assume that every non-constant
piecewise $C^1$-path in $M$ has positive $Q$-length. Then the path
pseudo-distance induced by $Q$ separates points locally.
\end{lem}

The proof of Lemma \ref{lem:semi-reduction}
is based on the
following theorem by Mityagin \cite{Mit15}.

\begin{thm}[Mityagin]\label{t:mityagin}
Let $V\subset\mathbb R^N$ be open and let
$F\colon V\to\mathbb R$
be a real-analytic function which is not identically zero. Then the zero set
$\{F=0\}$
is covered by a countable union of 
(not necessarily closed) real-analytic submanifolds of $V$.
\end{thm}

\begin{proof}[Proof of Lemma \ref{lem:semi-reduction}]
Since the statement is local, we work in a coordinate ball
$V\subset\mathbb R^N$.
Let $B(x)$ be the positive semi-definite symmetric matrix representing $Q$
in this chart, i.e.,
$Q(x,\xi)=\xi^T B(x)\xi$.
The entries of $B$ are real-analytic.

We first claim that
$\Delta(x):=\det B(x)$
is not identically zero on any open subset of $V$. Indeed, if
$\Delta\equiv0$ on some open set $W$, then $B(x)$ has a non-trivial kernel
for every $x\in W$. On a smaller open subset $W'\subset W$, the rank of
$B(x)$ is constant. Hence $\ker B(x)$ forms a non-trivial real-analytic
subbundle of $TW'$. Choosing a non-vanishing real-analytic vector field
$X(x)\in\ker B(x)$ and integrating it gives a non-constant analytic path
$\eta$ such that
$Q(\eta(t),\dot\eta(t))=0$
for all $t$. This path has zero length, contradicting the assumption.
Therefore $\Delta$ is not identically zero.

It follows that $Q$ is positive definite on
$V\setminus Z_1$ where $Z_1:=\{\Delta=0\}$.
By Theorem
\ref{t:mityagin}, $Z_1$ is covered by countably many real-analytic
submanifolds of positive codimension. Any zero-distance degeneracy class must
be contained in $Z_1$; otherwise it would meet a region where $Q$ is
positive definite, which 
would give positive
distance.

Now restrict $Q$ to each analytic submanifold $Y\subset Z_1$. The restricted
form is again real-analytic and positive semi-definite on $TY$. If its
determinant were identically zero on an open subset of $Y$, the same
constant-rank argument would produce a non-constant path inside $Y$ of zero
length, again contradicting the assumption. Therefore the determinant of the
restricted form is not identically zero. Applying 
Theorem  \ref{t:mityagin} again, 
its
zero locus inside $Y$ is covered by countably many analytic submanifolds of
strictly smaller dimension.

By induction, we can reduce 
the problem to the case of dimension $1$, i.e., an analytic path.
In this case,
a non-zero real-analytic
non-negative quadratic form has only isolated zeros unless it vanishes
identically; the latter is impossible by the positive-length assumption.
Therefore two distinct points in the same one-dimensional analytic stratum have
positive path distance. This
proves local separation and completes the proof.
\end{proof}

We can now conclude the 
proof of Theorem \ref{thm:path-metric}.
Define the path pseudo-distance
$$
d_{\rm cmc}(\lambda_1,\lambda_2):=
\inf_\gamma L_{\rm cmc}(\gamma),
$$
where the infimum is taken over all piecewise $C^1$-paths
$\gamma$ in $\Omega$ joining $\lambda_1$ to $\lambda_2$.

\begin{proof}[Proof of Theorem \ref{thm:path-metric}]
The function $d_{\rm cmc}$ is 
symmetric, non-negative, and
satisfies the triangle inequality, since it is defined as an infimum of path
lengths. Hence it is a pseudo-distance. We only need to prove that it separates
points.

By Proposition \ref{prop:analyticity-tangent-bundle}, the quadratic form $Q$ is real-analytic on the tangent
bundle. By Proposition \ref{prop:positive-length}, every non-constant piecewise $C^1$-path has
strictly positive length. Therefore Lemma \ref{lem:semi-reduction}
applies locally and shows 
the local separation of points.

Since $d_{\rm cmc}$ is induced by path lengths, local separation implies global separation on the connected component $\Omega$. Indeed, a path joining a point to a point outside a sufficiently small neighbourhood must cross the boundary of that neighbourhood, and the local separation gives a positive lower bound for the length needed to do so. Hence, we have $d_{\rm cmc}(\lambda_1,\lambda_2)>0$ whenever $\lambda_1\neq\lambda_2$.

Thus $d_{\rm cmc}$ is a distance on $\Omega$, and the proof is complete.
\end{proof}

\printbibliography
\end{document}